\documentclass[12pt]{article}
\usepackage{amsthm,amsfonts, amsbsy, amssymb,amsmath,graphicx}  
\usepackage{latexsym}
\usepackage{amssymb,amsmath,amsthm}
\usepackage{addfont}
\addfont{OT1}{rsfs10}{\rsfs}
\usepackage{multicol}
\usepackage{color}
\usepackage{cite}

\newtheorem{theorem}{Theorem}[section]
\newtheorem{corollary}[theorem]{Corollary}
\newtheorem{lemma}[theorem]{Lemma}
\newtheorem{observation}[theorem]{Observation}
\newtheorem{proposition}[theorem]{Proposition}

\DeclareMathOperator {\ex} {{\rm ex}}
\DeclareMathOperator {\diam} {{\rm diam}}

\title{Breaking Symmetry in Graphs by Resolving Sets}

\author{Meysam Korivand $^{a}$ \and Nasrin Soltankhah $^{a,}$\thanks{Corresponding author} \and Sandi Klav\v{z}ar $^{b,c,d}$ \\\\
$^{a}$\small Department of Mathematics, Faculty of Mathematical Sciences, \\
\small Alzahra University, Tehran, Iran \\
\small {\tt mekorivand@gmail.com, m.korivand@alzahra.ac.ir} \\
\small {\tt  soltan@alzahra.ac.ir} \\
$^{b}$\small Faculty of Mathematics and Physics, University of Ljubljana, Slovenia\\
\small {\tt sandi.klavzar@fmf.uni-lj.si} \\
$^{c}$\small Institute of Mathematics, Physics and Mechanics, Ljubljana, Slovenia \\
$^{d}$\small Faculty of Natural Sciences and Mathematics, University of Maribor, Slovenia \\
}

\date{\today}

\begin{document}

\maketitle

\begin{abstract} 
Let ${\rm dim}(G)$ and $D(G)$ respectively denote the metric dimension and the distinguishing number of a graph $G$. It is proved that $D(G) \le {\rm dim}(G)+1$ holds for every connected graph $G$. Among trees, exactly paths and stars attain the bound, and among connected unicyclic graphs such graphs are $t$-cycles for $t\in \{3,4,5\}$. It is shown that for any $1\leq n< m$, there exists a graph $G$ with $D(G)=n$ and ${\rm dim}(G)=m$. Using the bound $D(G) \le {\rm dim}(G)+1$, graphs with $D(G) = n(G)-2$ are classified. 
\end{abstract}


\medskip\noindent
\textbf{AMS Math.\ Subj.\ Class.\ (2020)}: 05C12, 05C15

%%%%%%%%%%%%%%%%%%%%%%%
\section{Introduction} 
\label{sec:intro}
%%%%%%%%%%%%%%%%%%%%%%%

Let $G = (V(G), E(G))$ be a graph and $S \subseteq V(G)$. Then $S$ is a {\it resolving set} if for any two vertices $g_1$  and $g_2$ of $G$, there exists a vertex $s \in S$ such that $d_G(g_1, s)\neq d_G(g_2, s)$, where $d_G$ is the standard shortest-path distance. The smallest size resolving set of $G$ is a {\em metric basis} of $G$, its size being the {\it metric dimension} (MD) ${\rm dim}(G)$ of $G$. 

Resolving sets were introduced independently and simultaneously by Slater \cite{slater} and Harary \& Melter \cite{harary} in 1975-6, and are one of the most widely studied and applied graph theory concepts. The 2023 survey~\cite{survey2} which presents an overview of the essential results
and applications of metric dimension contains 220 references, while the survey~\cite{survey1} that focuses on variants of metric dimension cites 203 papers. Among these variants of metric dimension we point to strong MD~\cite{b.1}, $k$-MD~\cite{c.4}, $k$-metric antidimension~\cite{2a}, local MD~\cite{b.2}, adjacency MD~\cite{c.5}, non-local MD~\cite{Klavzar0}, edge MD~\cite{c.7}, fractional MD~\cite{c.9}, and mixed MD~\cite{c.8}. It is also extremely important that various applications of metric dimension have been found so far to solve real-world problems such as privacy in social networks, error correcting codes, locating intruders in networks, chemistry, robot navigation, pattern recognition, image processing, and coin weighing, cf.~\cite{survey2}.

In 1977, Babai~\cite{babai} presented a concept whose goal is to break symmetry in graphs by a vertex partition of graphs with smallest size. Two decades later, in 1996, the concepts was reinvented by Albertson and Collins~\cite{alber} which led to a widespread attention of it. In fact, it was only a decade or so ago from now that it was noticed that the origin of the concept was Babai's paper; it thus went unnoticed for almost 40 years. Anyhow, a {\em distinguishing coloring} of a graph $G$ is a vertex coloring such that there is no color preserving non-trivial automorphism of $G$. The minimum number of colors required for a distinguishing coloring of $G$ is the {\it distinguishing number} $D(G)$ of $G$.

Distinguishing coloring of graphs have been investigated from many perspectives and with many different techniques, the papers~\cite{Ahmadi,  Collins2, Imrich, Russell} form a short selection of such efforts. Distinguishing coloring can be extended to arbitrary groups acting on sets, see~\cite{Babai, Bailey, Chan, Klavzar, Tymoczko}. Moreover, the concept has been the inspiration for several related concepts such as distinguishing maps~\cite{Tucker}, proper distinguishing coloring~\cite{Collins}, distinguishing arc-coloring~\cite{Kalinowski2}, distinguishing index~\cite{Kalinowski}, and distinguishing threshold~\cite{Shekarriz}. 

The metric dimension and the distinguishing number were investigated for years without paying attention to each other. The starting observation of this paper is that solving the metric dimension problem for a graph enables to break its symmetries. We note in passing that this is not the first occasion on which the concept of metric dimension has been linked with automorphism-dependent parameters, such as in determining sets~\cite{Det_2, Det_1}. We proceed as follows.

In the next section we list some definitions needed, recall a few results, and deduce a couple of new ones to be used later on. In Section~\ref{sec:connection} we first show that $D(G) \le \dim(G)+1$ holds for every connected graph $G$. After that we prove that among trees, exactly paths and stars attain the bound, while among the connected unicyclic graphs such graphs are $t$-cycles for $t\in \{3,4,5\}$. We also show that for any $1\leq n< m$, there exists a graph $G$ with $D(G)=n$ and $\dim(G)=m$.

Almost every graph has trivial automorphism group, hence $D(G) = 1$ for almost all graphs. In addition, if a graph has a nontrivial automorphism group, then its distinguishing number is usually small, typically $2$. In view of the bound $D(G) \le \dim(G)+1$, it is natural to search for graphs with large distinguishing number, we do it in Section~\ref{sec:large-D}. In this direction we respectively classify the graphs $G$ with $D(G) = n(G)$, $D(G) = n(G)-1$, and $D(G) = n(G)-2$, where $n(G)$ denotes the order of $G$. 

%%%%%%%%%%%%%%%%%%%%%%%%%%%%%%%%%%%%%%%%%%% 
\section{Preliminaries} 
\label{sec:prelim}
%%%%%%%%%%%%%%%%%%%%%%%%%%%%%%%%%%%%%%%%%%% 

Let $G = (V(G), E(G))$ be a graph. The neighborhood of a vertex $u\in V(G)$, that is, the set of its neighbors, is denoted by $N_G(u)$. The diameter $\diam(G)$ of $G$ is the largest distance between the vertices of $G$. The complement of $G$ is denoted by $\overline{G}$. The join $G+H$ of graphs $G$ and $H$ is obtained from the disjoint union of $G$ and $H$ by making adjacent every vertex of $G$ with every vertex of $H$. The {\em kite} is a graph on five vertices obtained from a $K_4-e$ by adding a vertex and making it adjacent to one vertex of degree $2$. If $k$ is a positive integer, then $[k]$ denotes the set $\{1,\ldots, k\}$. 

If $t\ge 1$, then we define 
\begin{align*}
{\cal C}_t & = \{G:\ G\ \text{is a connected graph with}\ D(G) = t\}\,,\\
\overline{{\cal C}}_t & = \{G:\ G\ \text{is a disconnected graph with}\ D(G) = t\}\,.
\end{align*}

\begin{lemma} \label{lem:discon}
If $t\ge 1$, then $\overline{{\cal C}}_t = \{G:\ \overline{G}\in {\cal C}_t\} \setminus {\cal C}_t$.
\end{lemma}

\begin{proof} 
Assume first that $H\in \overline{{\cal C}}_t$. As $H$ is disconnected, $\overline{H}$ is connected. Since also $D(H) = D(\overline{H}) = t$ we get $\overline{{\cal C}}_t \subseteq \{G:\ \overline{G}\in {\cal C}_t\} \setminus {\cal C}_t$. 

Assume second that $H\in \{G:\ \overline{G}\in {\cal C}_t\} \setminus {\cal C}_t$. Since $\overline{H}\in {\cal C}_t$, we have $D(H) = t$. Moreover, $H$ must be disconnected since $H\notin {\cal C}_t$. Thus $\{G:\ \overline{G}\in {\cal C}_t\} \setminus {\cal C}_t \subseteq \overline{{\cal C}}_t$. 
\end{proof}

\begin{theorem} {\rm \cite[Theorem 3]{0}}
\label{t2}
If $G$ is a connected graph with $n(G) \ge 2$, then $\dim(G) = n(G) - 1$ if and only if $G \cong K_{n(G)}$.
\end{theorem}

\begin{theorem} {\rm \cite[Theorem 4]{0}}
\label{t3}
If $G$ is a connected graph with $n(G) \geq 4$, then $\dim(G) = n(G) - 2$ if and only if $G= K_{s,t}$ $(s,t\geq 1)$, $G=K_s + \overline{K_t}$ $(s\geq 1, t\geq 2)$, or $G= K_s + (K_1 \cup K_t)$ $(s,t \geq 1)$. 
\end{theorem}

\begin{proposition}\label{prop:lemn-2}
Let $G$ be a connected graph with $\dim(G) = n(G) - 2 \ge 2$. If $\ell \geq 1$, then $D(G)=n(G)-\ell$ if and only if $G$ is one of the following graphs: 
\begin{itemize}
\begin{small}
\begin{multicols}{2}
\item[(a)] 
$K_{\ell+1, \ell+1}$
\item[(c)] 
$K_{\ell}+\overline{K_t}$, $t\geq \ell$
\item[(e)] 
$K_{\ell-1}+(K_t \cup K_{1})$, $t\geq \max \{2, \ell-1\}$
\item[(b)] 
$K_{t, \ell}$, $t\geq \ell+1$
\item[(d)] 
$K_t + \overline{K_{\ell}}$, $t\geq \ell\geq 2$
\item[(f)]
$K_{t}+(K_{\ell-1}\cup K_{1})$, $t\geq \max \{2, \ell-1\}$.
\end{multicols}
\end{small}
\end{itemize}
\end{proposition}

\begin{proof}
Let $\ell \geq 1$. As we have assumed that $\dim(G) = n(G) - 2 \ge 2$, we need to check which of the graphs listed in Theorem~\ref{t3} fulfill the condition $D(G)=n(G)-\ell$. 

Assume first that $G= K_{s,t}$ $(s,t\geq 1)$. 
If $s=t$, then $D(G)=t+1$. The assumption $D(G)=n-\ell$ implies that $t+1=s+t-\ell$. So, $s=t=\ell+1$.  This means that $G= K_{\ell+1, \ell+1}$. If $s\neq t$, then $D(G)=\max \{s, t\}$. 
Let $\max \{s, t\}=t$. Then $t=s+t-\ell$, and so $s=\ell$. Hence, $G=K_{t,\ell}$ 
for $t\geq \ell+1$.  

Let $G=K_s + \overline{K_t}$ $(s\geq 1, t\geq 2)$. Then we have $D(G)=\max \{s, t\}$. If $\max \{s, t\}=t$, then $t=s+t-\ell$ and so $s=\ell$. 
This implies that $G=K_{\ell}+ \overline{K_t}$. And if $\max \{s, t\}=s$, then $t=\ell$ and $G$ is the graph as stated in (d).

Finally, suppose that $G= K_s + (K_1 \cup K_t)$ $(s,t \geq 1)$. 
Then we have $D(G)=\max \{s, t\}$. If $\max \{s, t\}=t$ ($\max \{s, t\}=s$), then $s=\ell -1$ ($t=\ell -1$). In such a situation, we obtain the graphs as stated in (e) and (f).
\end{proof}

%%%%%%%%%%%%%%%%%%%%%%%%%%%%%%%%%%%%%%%%%%% 
\section{The connection} 
\label{sec:connection}
%%%%%%%%%%%%%%%%%%%%%%%%%%%%%%%%%%%%%%%%%%% 

In this section we first show the basic connection between the distinguishing number and the metric dimension, that is, if $G$ is a connected graph, then $D(G)\le \dim(G) + 1$. Trees and unicyclic graphs which attain the equality are characterized. We also demonstrate that for any $1\leq n< m$, there exists a graph $G$ with $D(G)=n$ and ${\rm dim}(G)=m$.

\begin{proposition}\label{prop:main}
If $G$ is a connected graph, then $D(G)\leq \dim(G)+1$. Moreover, the bound is sharp. 
\end{proposition}

\begin{proof}
Let $S \subseteq V(G)$ be a metric basis of $G$, so that $|S| = \dim(G)$. Let $c$ be a coloring of $V(G)$ which respectively assigns colors from $[\dim(G)]$ to the vertices from $S$, and assigns color $\dim(G)+1$ to each vertex from $V(G)\setminus S$. We claim that $c$ is a distinguishing coloring for $G$. Suppose on the contrary that there exists a non-trivial automorphism $\varphi$ that preserves $c$. Then $\varphi(u) = u$ for each $v\in S$, hence, since $\varphi$ is non-trivial, there exist vertices $u,v \in V(G)\setminus S$, $u\ne v$, such that $f(u)=v$. Since an automorphism is a distance preserving mapping, for every $w \in S$ we have 
$$ d_G(u, w) = d_G(f(u), f(w)) = d_G(v, w)\,.$$ 
But then $S$ does not resolve $u$ and $v$, a contradiction.

To see that the bound is sharp, consider paths $P_n$, $n\geq 2$, for which we have $D(P_n) = 2$ and $\dim(G) = 1$.
\end{proof}

As observed in the proof of Proposition~\ref{prop:main}, paths attain the equality in Proposition~\ref{prop:main}. The equality is also attained by complete graphs $K_n$, $n\ge 2$, since $D(K_n) = n$ and $\dim(K_n) = n-1$. A sporadic example for the equality is the wheel graph $W_5$ for which we have $D(W_5) = 3$ and $\dim(W_5) = 2$, see~\cite{shan-2002}. We next show that besides paths, stars are the only trees that attain the equality. To prove it, the following definitions will be useful. 

Let $T$ be a tree. A vertex $u\in V(T)$  is a {\em leaf} if $\deg_T(u) = 1$, and a {\em branch} if $\deg_T(u) \ge 3$. The {\em terminal degree} of a branch $u$ is the number of leaves $x$, such that $d_T(x, u) < d_T(x, v)$ for all branches $v\ne u$. A branch is {\em external} if its terminal degree is strictly more than one, otherwise it is {\em internal}. The number of leaves and of external branches of $T$ are respectively denoted by $\ell (T)$ and $\ex(T)$. Using these definitions, the key result for the metric dimension of trees asserts that if $T$ is a tree different from a path, then 
\begin{equation}
\label{eq:did-of-trees}
\dim(T) = \ell(T) - \ex(T)\,,  
\end{equation}
see~\cite{harary, Khuller, slater}.
We further say that the branch closest to a leaf is the {\em ancestor} of the leaf. The leaf is {\em external} if its ancestor is external. Finally, the path between a leaf and its ancestor will be called a {\em leg}. 

\begin{theorem}
\label{thm:trees}
A tree $T$ satisfies the equality $D(T) = \dim(T)+1$ if and only if $T \in\{P_n, K_{1,n}:\ n\ge 2\}$. 
\end{theorem}

\begin{proof}
If $n\ge 2$, then $\dim(P_n) = 1$, $D(P_n) = 2$, $\dim(K_{1,n}) = n-1$, and $D(K_{1,n}) = n$. It remains to prove that for any tree $T \notin\{P_n, K_{1,n}:\ n\ge 2\}$ we have $D(T) \ne \dim(T)+1$. Since $\dim(P_1) = D(P_1) = 1$, we may assume in the rest of the proof that $|V(T)|\ge 3$. 

Assume first that $\ex(T) = 1$ and let $b$ be the unique external branch. Set $d = \deg_T(b)$. In this case, from~\eqref{eq:did-of-trees} we get $\dim(T) = d-1$. Recall that $T$ is not a star, and let $p$ be an arbitrary leaf of $T$ and $p'$ its neighbor. Color leaves different from $p$ with distinct colors from $[d-1]$. Assign color $2$ to $p'$, and color $1$ to all the remaining vertices, including $p$. This is a distinguishing coloring of $T$, hence $D(T) \le d - 1 = \dim(T)$. 

Assume second that $\ex(T) = k\ge 2$, and let $b_1, \ldots, b_k$ be the external branches of $T$. Let 
$$t=\max \left\{k, \max_{i\in [k]} \deg_T (b_i) - 1 \right\}\,.$$ 
We claim that $\dim(T)\ge t$. First, by~\eqref{eq:did-of-trees} we have $\dim(T)\ge k$. Assume next that, without loss of generality, $\deg_T (b_1) = \max_{i\in [k]} \deg (b_i)$. Let $\deg_T (b_1) = d_1 + d_2$, where $d_1$ is the number of legs containing $b_1$. Then in each such leg except one, there exists a member in any metric basis of $T$. Moreover, in view of~\eqref{eq:did-of-trees}, each neighbor of $b_1$ not a leg contributes at least $1$ to the metric dimension of $T$. This means that 
$$\dim(T) \ge (d_1 - 1) + d_2 = \deg_T (b_1) - 1 = \max_{i\in [k]} \deg (b_i) - 1\,.$$ 
We have thus proved that $\dim(T)\ge t$. Now, for $i\in [k]$, color $b_i$ with color $i$. Additionally, color the external leaves of the same ancestor $b_i$ with distinct colors from the set $[\deg_T(b_i)]$.  
(Note that since there are at least two external branches, there can be no more than $\deg_T(b_i)-1$ external leaves corresponding to the same ancestor $b_i$). Finally, assign color $1$ to all the remaining vertices. This coloring is a distinguishing coloring which proves that $D(T) \leq t$ and we are done. 
\end{proof} 

We next determine the graphs among unicyclic graphs which attain the equality in Proposition~\ref{prop:main}.

\begin{theorem}
\label{thm:unicyclic}
If $G$ is a connected unicyclic graph, then $D(G) = \dim(G)+1$ if and only if $G \in \{C_3, C_4, C_5\}$.
\end{theorem}

\begin{proof}
Let $G$ be a unicyclic graph and let $t_1, \ldots, t_m$ be the consecutive vertices of the unique cycle $C$ of $G$. Let further $T_i$, $i\in [m]$, be the maximal connected subgraph of $G$ which contains $t_i$ and no other vertex of $C$. Then $T_i$ is a tree, where it is possible that $T_i\cong K_1$. In particular, if $G$ is a cycle, then every $T_i$ is the one vertex graph. We now distinguish two cases. 

In the first case every $T_i$, $i\in [m]$, is isomorphic to a path. Then every automorphism of $G$ restricted to $C$ is an automorphism of $C$. It follows that if $m \geq 6$, then $D(G)=2$ and hence $D(G) < \dim(G) + 1$. Consider next the cases $m\in \{3,4,5\}$. Assume that each path $T_i$ has at least two vertices and let $t'_i$ be the neighbor of $t_i$ in $T_i$. If $m = 3$, then assign color $1$ to $t_1$, $t'_1$ and $t_3$, and color $2$ to $t_2$, $t'_2$ and $t'_3$. If $m = 4$, then assign color $1$ to $t_1$, $t'_1$, $t_3$ and $t'_4$, and color $2$ to $t_2$, $t'_2$, $t'_3$ and $t_4$. And if $m = 5$, then assign color $1$ to $t_1$, $t'_1$, $t_3$, $t'_3$, $t'_4$ and $t_5$, and color $2$ to $t_2$, $t'_2$, $t_4$ and $t'_5$. Finally, in each of the three cases assign color $1$ to all the other vertices. In each of the cases we have a distinguishing coloring, so $D(G)\le 2$ and thus $D(G) \leq {\rm dim}(G)$. Note that the above argument is also applicable as soon as at least one $T_i$ has at least two vertices. Hence we can conclude that if $m\in \{3,4,5\}$, then $D(G) = \dim(G) + 1$ if and only if $G\in \{C_3, C_4, C_5\}$. 

In the second case at least one $T_i$ is not a path, we may assume without loss of generality that $T_1$ is such a tree. Let $S=\{s_1, \ldots, s_k\}$ be a metric basis of $G$.  Since $T_1$ is not a path, we have $S \cap V(T_1) \neq \emptyset$, and we may  assume that $s_1 \in S \cap V(T_1)$. Color now the vertices $s_i$, $i\in [k]\setminus \{1\}$, with color $i$, and all the other vertices, that is, vertices from $(V(G)\setminus S) \cup \{s_1\}$, with color $1$. If a non-trivial automorphism $f$ preserves this coloring, then $f(s_i) = s_i$ for $i\in [k]\setminus \{1\}$ and $f(s_1)\in V(G)\setminus S$. Since $t_2$ and $t_m$ are not resolved by vertices from $S \cap V(T_1)$, there exists a vertex in $S \setminus V(T_1)$, say $s_2$. By \cite[Corollary 7]{Hakanen} we may assume that $t_1 \notin S$, so by now it is colored $1$. Change now the color of $s_1$ and of $t_1$ to $2$, so that the color class of color $2$ contains the vertices $s_1, s_2$, and $t_1$. This is now a distinguishing coloring using $k = {\rm dim}(G)$ colors and we can conclude that $D(G) \le \dim(G)$.
\end{proof}

To conclude the section we demonstrate the following result which complements Proposition~\ref{prop:main}. 

\begin{proposition}\label{w}
For any $1\leq n< m$, there exists a graph $G$ with $D(G)=n$ and ${\rm dim}(G)=m$.
\end{proposition} 

\begin{proof}
If $k\geq 3$, then let $T_k$ be the tree obtained from $K_{1,k}$ by respectively subdividing its edges $0,1,\ldots, k-1$ times. Then $T_k$ is asymmetric, so that $D(T_k) = 1$. On the other hand, from~\eqref{eq:did-of-trees} we get $\dim(T_k) = k-1$. This settles the case $n = 1$. 

Assume now that $2\leq n < m$. Let $G_{n,m}$ be the graph obtained from the disjoint union of $T_{m-n+2}$ and $K_n$ by adding the edges between the maximum degree vertex of $T_{m-n+2}$ and all the vertices of $K_n$. Then $D(G_{n,m})=n$ and $\dim(G_{n,m})=m$. 
\end{proof} 

%%%%%%%%%%%%%%%%%%%%%%%%%%%%%%%%%%%%%%%%%%%%%
\section{Graphs $G$ with $D(G)$ close to $n(G)$} 
\label{sec:large-D}
%%%%%%%%%%%%%%%%%%%%%%%%%%%%%%%%%%%%%%%%%%%%%

In this section we classify graphs $G$ such that $D(G) \in \{n(G), n(G)-1, n(G)-2\}$. First, combining Proposition~\ref{prop:main} with Theorem~\ref{t2} and having Lemma~\ref{lem:discon} in mind we get: 

\begin{corollary}\label{cor:r1}
If $G$ is a graph, then $D(G) = n(G)$ if and only if $G \in \{K_{n(G)}, \overline{K}_{n(G)} \}$.
\end{corollary} 

Using the classification of graphs $G$ with $\dim(G) = n(G) - 2$ from Theorem~
\ref{t3} and its application in Proposition~\ref{prop:lemn-2}, the graphs with $\dim(G) = n(G) - 1$ are the following.

\begin{theorem}\label{then-2}
If $G$ is a graph, then $D(G) = n(G) - 1$ if and only if $G$ is one of $C_4$, 
$2K_{2}$, $K_{t, 1}$, and $K_{t}\cup K_1$, where $t\geq 2$.
\end{theorem}

\begin{proof} 
Assume first that $G$ is a connected graph. By Proposition \ref{prop:main}, $\dim(G) \in \{n(G)-2, n(G)-1\}$. Moreover, by Theorem~\ref{t2} and Corollary~\ref{cor:r1}, it more specifically holds that $\dim(G) \neq n(G)-1$, that is, $\dim(G)= n(G)-2$. To determine such graphs, set $\ell=1$ in Proposition~\ref{prop:lemn-2}. Then $C_4$ and the graphs $K_{t, 1}$, $t\geq 2$, are obtained from Proposition~\ref{prop:lemn-2} (a) and (b). Since $\ell =1$, $G$ is not equal to any graph in Proposition~\ref{prop:lemn-2} (d), (e) and (f). The graphs $2K_{2}$, and $K_{t}\cup K_1$, $t\geq 2$, are then deduced by Lemma~\ref{lem:discon}. 
\end{proof} 

In the rest of the section we classify the graphs with $\dim(G) = n(G) - 2$. To do so, we use Proposition~\ref{prop:main} together with the forthcoming Theorems~\ref{Jannesari} and \ref{Hernando} in which graphs of metric dimension relevant to us and of given diameter are described. For stating the next theorem, some preparation is needed. 

Vertices $u$ and $v$ of a graph $G$ are {\it twins} if ${\rm N}_G(v) \setminus \{u\}={\rm N}_G(u) \setminus \{u\}$. Define the relation $\equiv$ on $V(G)\times V(G)$ by setting $u \equiv v$ if $u = v$, or if $u$ and $v$ are twins. Then $\equiv$ is an equivalence relation. For a vertex $v\in V(G)$ we denote its equivalence class by $v^{\ast}$, that is, 
$$v^{\ast} = \{u \in V(G):\ u \equiv v\}\,.$$ 
The {\em twin graph} $G^{\ast}$ of $G$ is the quotient graph with respect to the relation  $\equiv$, that is, 
$$V(G^{\ast}) = \{v^{\ast}:\ v \in V(G)\}\quad {\rm and}\quad E(G^{\ast}) = \{v^{\ast} u^{\ast}:\ uv \in E(G)\}\,.$$  
A twin equivalence class $v^{\ast}$ of $G$ induces an edgeless graph or a complete graph. This subgraph is denoted by $G[v^{\ast}]$. We will further say that $v^{\ast}$ is of 
\begin{itemize}
\item {\em type (1)}, if $G[v^{\ast}] \cong K_1$, 
\item {\em type (K)}, if $G[v^{\ast}] \cong K_2$, where $r\ge 2$, 
\item {\em type (N)}, if $G[v^{\ast}] \cong \overline{K_r}$, where $r\geq 2$.
\end{itemize}
Furthermore, we say that $v^{\ast} \in G^{\ast}$ is of {\em type (1K)} if $v^{\ast}$
is of type (1) or (K), of {\em type (1N)} if $v^{\ast}$ is of type (1) or (N), of {\em type (KN)} if $v^{\ast}$ is of type (K) or (N), and of {\em type (1KN)} if 
$v^{\ast}$ is of type (1), (K) or (N). For more details about twin graphs see~\cite{Hernando, Jannesari}. Now we can recall the following result. 

\begin{theorem} {\rm \cite[Theorem 1]{Jannesari}}
\label{Jannesari}
Let $G$ be a connected graph with $\diam(G) = 2$. Then ${\rm dim}(G)= n(G)-3$ if and only if $G^{\ast}$ is one of the following graphs.
\begin{itemize}
\item[{\bf G$_1$}.] 
$G^{\ast}= K_{3}$ and has at most one vertex of type (1K);
\item[{\bf G$_2$}.] 
$G^{\ast}= P_3$ and one of the following cases holds:
\begin{itemize}
\item[(a)] 
The degree-$2$ vertex is of type (N), and one of the leaves is of type (K) and the other is of any type;
\item[(b)] 
One of the leaves is of type (K), the other is of type (KN) and the degree-$2$ vertex is of any type;
\end{itemize}
\item[{\bf G$_3$}.] 
$G^{\ast}$ is a triangle with a pendant edge, one of the degree-$2$ vertices is of type (N), the other is of type (1K), and the leaf is of type (1N). Moreover, a degree-$2$ vertex of type (K) yields the leaf and the degree-$3$ vertex are not of type (N);
\item[{\bf G$_4$}.] 
$G^{\ast}= C_5$, and each vertex is of type (1);
\item[{\bf G$_5$}.] 
$G^{\ast}$ is a $C_5$ with a chord, adjacent degree-$2$ vertices are of type (1), the other
vertices are of type (1K);
\item[{\bf G$_6$}.] 
$G^{\ast}=K_1 + P_4$, the degree-$4$ vertex is of any type, the others are of type (1K). Furthermore, two non-adjacent vertices are not of type (K), and two adjacent
vertices are not of different types (K) and (N);
\item[{\bf G$_7$}.] 
$G^{\ast}$ is a kite with a pendant edge adjacent to a degree-$3$ vertex, the leaf is of type (1), the degree-$4$ and degree-$3$ vertices are type (1K), one of the degree-$2$ vertices is of type (K) and the other is of type (1);
\item[{\bf G$_8$}.] 
$G^{\ast}$ is a kite, one of the degree-$2$ vertices is of type (K), the other is of type (1), one of the degree-$3$ vertices is of type (N), and the other is of type (1K);
\item[{\bf G$_9$}.] 
$G^{\ast}= C_4$, two adjacent vertices are of type (K), the others are of type (1); 
\item[{\bf G$_{10}$}.] 
$G^{\ast}= C_4 + K_1$, two degree-$3$ adjacent vertices are of type (K), degree-$4$ vertex is of type (1K), others are of type (1).
\end{itemize}
\end{theorem} 

\begin{observation} \label{obser}
Let $m$ be a positive integer, $G$ a connected graph, and $V(G^{\ast})=\{v^{\ast}_1, \ldots, v^{\ast}_{n(G^{\ast})}\}$. If $D(G) =|v^{\ast}_i|$  for some $i\in [n(G^{\ast})]$ and $\Sigma_{j\in n(G^{\ast})\atop j\ne i} |v^{\ast}_j| \ne m$, 
then $D(G) \neq n(G)-m$.
\end{observation}

We say that a graph $G$ is {\it almost asymmetric} if there is no non-trivial automorphism that maps a vertex from one twin equivalence class of $G$ to a vertex of another twin equivalence class of. Clearly, if $G$ is almost asymmetric, then $D(G)= \max\limits_{v^{\ast} \in V(G^{\ast})} \{|v^{\ast}|: \ v^{\ast} \in V(G^{\ast}) \}$. The following result directly follows from Observation~\ref{obser}, we state it for later reference. 

\begin{corollary}\label{lem-almost-asymmetric}
Let $m$ be a positive integer and $G$ a connected, almost asymmetric graph. If $n(G) - \max\limits_{v^{\ast} \in V(G^{\ast})} \{|v^{\ast}|:\ v^{\ast} \in V(G^{\ast}) \}  \neq m$, then $D(G) \neq n(G)-m$.
\end{corollary} 

In the following six lemmas we investigate the distinguishing number of the graphs from Theorem~\ref{Jannesari}.

\begin{lemma}\label{G1}
If $G^{\ast}$ is a {\bf G}$_1$ graph, then $D(G) =n(G)-2$ if and only if $G$ is $K_{1, 2, 2}$ or $K_{1, 1, t}$, $t\geq 2$. 
\end{lemma}

\begin{proof} 
Assume first that  $v^{\ast}_1$ is a vertex of $G^{\ast}$ of type (1). Then the other two vertices $v^{\ast}_2$ and $v^{\ast}_3$ of $G^{\ast}$ are of type (N). If 
$|v^{\ast}_2|\neq |v^{\ast}_3|$, then $D(G)=\max \{|v^{\ast}_2|, |v^{\ast}_3|\}$. 
Let $\max \{|v^{\ast}_2|, |v^{\ast}_3|\}=|v^{\ast}_2|$. 
Hence $D(G)=n(G)-2$ implies that $|v^{\ast}_3|=1$, and we get the graphs $K_{1, 1, t}$, $t\geq 2$.  
If $|v^{\ast}_2|=|v^{\ast}_3|$, then $D(G)=|v^{\ast}_2|+1$. 
The equation $|v^{\ast}_2|+1=|v^{\ast}_2|+|v^{\ast}_3|+1-2$ 
results in $|v^{\ast}_3|=2$, and we obtain $K_{1, 2, 2}$. 

Let $v^{\ast}_1$ be of type (K). 
If $|v^{\ast}_2|\neq |v^{\ast}_3|$, then $D(G)=\max \{|v^{\ast}_1|, |v^{\ast}_2|, |v^{\ast}_3|\}$. Without loss of generality, we may assume that 
$\max \{|v^{\ast}_1|, |v^{\ast}_2|, |v^{\ast}_3|\}=|v^{\ast}_1|$. 
The assumption $D(G)=n(G)-2$ yields $|v^{\ast}_2|+|v^{\ast}_3|=2$, 
which is impossible. 
If $|v^{\ast}_2|= |v^{\ast}_3|$, 
then $D(G)=\max \{|v^{\ast}_1|, |v^{\ast}_2|+1\}$. 
If $D(G)=|v^{\ast}_1|$, then the size of $v^{\ast}_2$ will be equal to 1, which is impossible. 
If $D(G)=|v^{\ast}_2|+1$, then $|v^{\ast}_1| + |v^{\ast}_3|=3$, which is again not impossible.

Assume that $v^{\ast}_1$ is of type (N). 
If $|v^{\ast}_1|\neq |v^{\ast}_2|\neq |v^{\ast}_3| \neq |v^{\ast}_1|$, 
then $D(G)=\max \{|v^{\ast}_1|, |v^{\ast}_2|, |v^{\ast}_3|\}$. 
Let $\max \{|v^{\ast}_1|, |v^{\ast}_2|, |v^{\ast}_3|\}=|v^{\ast}_1|$. 
Hence, $|v^{\ast}_2|+|v^{\ast}_3|=2$, which is impossible.  
If $|v^{\ast}_1|=|v^{\ast}_2|\neq |v^{\ast}_3|$, then
$D(G)=\max \{|v^{\ast}_1|+1, |v^{\ast}_3|\}$. 
If $D(G)=|v^{\ast}_3|$, then $|v^{\ast}_1| =1$, 
which contradicts the definition of type (N).
If $D(G)=|v^{\ast}_1|+1$, then $D(G)\neq n-2$. If $|v^{\ast}_1|=|v^{\ast}_2|= |v^{\ast}_3|$,  then
$D(G)=|v^{\ast}_1|+1$ and $|v^{\ast}_2|+|v^{\ast}_3|=3$, which is impossible.
\end{proof}

\begin{lemma}\label{G2}
If $G^{\ast}$ is a {\bf G}$_2$ graph, then $D(G)\neq n(G)-2$.
\end{lemma}

\begin{proof} 
Suppose on the contrary that $D(G)= n(G)-2$. 

Assume first that $G^{\ast}$ is a {\bf G}$_2$(a) graph. Let $v^{\ast}_1$, $v^{\ast}_2$ and $v^{\ast}_3$ be the vertices of types (K), (N), and any type, respectively. Suppose first that $v^{\ast}_3$ 
is of type (1). Thus $D(G)=\max \{|v^{\ast}_1|, |v^{\ast}_2|\}$. 
If  $\max \{|v^{\ast}_1|, |v^{\ast}_2|\}= |v^{\ast}_1|$, then 
$|v^{\ast}_1|= |v^{\ast}_1|+|v^{\ast}_2| -1$ and so $|v^{\ast}_2|= 1$, a contradiction. 
If $\max \{|v^{\ast}_1|, |v^{\ast}_2|\}= |v^{\ast}_2|$, then using a similar argument as before, $|v^{\ast}_1|= 1$, a contradiction.
Let $v^{\ast}_3$ be of type (N). 
Then $D(G)=\max \{|v^{\ast}_1|, |v^{\ast}_2|, |v^{\ast}_3|\}$. Without loss of generality, we may assume that $\max \{|v^{\ast}_1|, |v^{\ast}_2|, |v^{\ast}_3|\}=|v^{\ast}_1|$. So, $|v^{\ast}_2|+|v^{\ast}_3|=2$, which is a contradiction. 
Assume that $v^{\ast}_3$ is of type (K). 
Clearly, $|v^{\ast}_1|=|v^{\ast}_3|$. So, $D(G)=\max \{|v^{\ast}_1|+1, |v^{\ast}_2|\}$. If $\max \{|v^{\ast}_1|+1, |v^{\ast}_2|\}=|v^{\ast}_2|$, then it is impossible as before. If $\max \{|v^{\ast}_1|+1, |v^{\ast}_2|\}=|v^{\ast}_1|+1$, then $|v^{\ast}_2|+ |v^{\ast}_3|=3$, a contradiction. 

Assume second that $G^{\ast}$ is a {\bf G}$_2$(b) graph. Let $v^{\ast}_1$, $v^{\ast}_2$ and $v^{\ast}_3$ be the vertices of types (K), any type, and (NK), respectively. 
If $v^{\ast}_3$ is of type (N), then $|v^{\ast}_1|= 1$, which is impossible. 
Let $v^{\ast}_3$ be of type (K). If $|v^{\ast}_1|\neq |v^{\ast}_3|$, 
then we reach the same contradiction as before. Thus 
$|v^{\ast}_1| = |v^{\ast}_3|$ and so 
$D(G)=\max \{|v^{\ast}_1|+1, |v^{\ast}_2|\}$. One can check that 
$D(G)\neq|v^{\ast}_2|$. So, $D(G)=|v^{\ast}_1|+1$. 
This concludes that $|v^{\ast}_2|+ |v^{\ast}_3|=3$, a contradiction.  
\end{proof}

\begin{lemma}\label{G3}
If $G^{\ast}$ is a {\bf G}$_3$ graph, then $D(G)\neq n(G)-2$. 
\end{lemma}

\begin{proof} 
Let $v^{\ast}_1, v^{\ast}_2, v^{\ast}_3$, and $v^{\ast}_4$ be the degree-$2$ vertex of type (1K), the degree-$2$ vertex of type (N), the degree-$3$ vertex, and the leaf, respectively. Suppose that $D(G)= n(G)-2$. Then in all cases, $G$ is almost asymmetric. Hence, $D(G)=\max \{|v^{\ast}_1|, |v^{\ast}_2|, |v^{\ast}_3|, |v^{\ast}_4|\}$. Then from $D(G)=n-2$ it follows that some of $|v^{\ast}_i|$'s must be equal to $1$ or $0$, which is impossible. 
\end{proof}

\begin{lemma}\label{G5}
If $G^{\ast}$ is a {\bf G}$_5$ graph, 
then $D(G)\neq n(G)-2$.
\end{lemma}

\begin{proof}
Suppose that $D(G)= n(G)-2$.
Let $v^{\ast}_1$ and $v^{\ast}_2$ be the degree-$3$ vertices, $v^{\ast}_3$ 
be the degree-$2$ vertex that is adjacent to $v^{\ast}_1$ and $v^{\ast}_2$, 
and $v^{\ast}_4$ and $v^{\ast}_5$ be the adjacent degree-$2$ vertices. Let $|v^{\ast}_1|=|v^{\ast}_2|=1$. 
If $|v^{\ast}_3|=1$, then $D(G)= 2$. 
If $|v^{\ast}_3|\neq 1$, then $D(G)= |v^{\ast}_3|=n(G)-4$, a contradiction. If $|v^{\ast}_1|=|v^{\ast}_2|\neq 1$ or $|v^{\ast}_1|\neq |v^{\ast}_2|$, 
then $D(G)=\max \{|v^{\ast}_1|, |v^{\ast}_2|, |v^{\ast}_3|\}$. 
Therefore, in all cases $D(G)\neq n(G)-2$.
\end{proof}

\begin{lemma}\label{G6}
If $G^{\ast}$ is a {\bf G}$_6$ graph, then $D(G)\neq n(G)-2$.  
\end{lemma}
\begin{proof}
Let $v^{\ast}_1$ and $v^{\ast}_4$ be the degree-$2$ vertices, 
$v^{\ast}_3$ and $v^{\ast}_5$ be the degree-$3$ vertices, and $v^{\ast}_2$ 
be the degree-$4$ vertex. 
Suppose that $D(G)= n(G)-2$. 
The only non-trivial automorphism of $G^{\ast}$ is the automorphism which fixes $v^{\ast}_2$, swaps $v^{\ast}_1$ with $v^{\ast}_4$, and swaps $v^{\ast}_3$ with $v^{\ast}_5$.
Since the non-adjacent vertices are not of type (K), if $|v^{\ast}_1|=|v^{\ast}_4|$, then $|v^{\ast}_1|=1$. 
Now, if $|v^{\ast}_3|=|v^{\ast}_5|=1$, then $D(G)=\max \{2, |v^{\ast}_2|\}$. 
If $D(G)=2$, then some vertices of $G^{\ast}$ have size $0$ which is clearly not possible. If $D(G)=|v^{\ast}_2|$, then $|v^{\ast}_2|+2=|v^{\ast}_2|$, a contradiction. 
If $|v^{\ast}_3|=|v^{\ast}_5|\neq 1$, then $D(G)=\max \{|v^{\ast}_3|, |v^{\ast}_2|\}$. In both cases $D(G)= |v^{\ast}_3|$ and $D(G)=|v^{\ast}_2|$, so$D(G)\neq n(G)-2$. If $|v^{\ast}_3|\neq |v^{\ast}_5|$ or $|v^{\ast}_1| \neq |v^{\ast}_4|$, then $G$ is an almost asymmetric graph and Corollary~\ref{lem-almost-asymmetric} yields the conclusion. 
\end{proof}

\begin{lemma} \label{Gal}
If $G^{\ast}$ is one of the graphs {\bf G}$_7$, {\bf G}$_8$, {\bf G}$_9$, {\bf G}$_{10}$, then $D(G)\neq n(G)-2$.
\end{lemma}

\begin{proof}
{\bf G}$_7$ is an almost asymmetric graph. Setting $m=2$ in Corollary~\ref{lem-almost-asymmetric} we can conclude that $D(G) \neq n(G)-m$. 

If $G^{\ast} \in \{${\bf G}$_8$, {\bf G}$_9\}$, then $D(G)= \max\limits_{v^{\ast} \in V(G^{\ast})} \{|v^{\ast}|:\ v^{\ast} \in G^{\ast} \}$. 
So, $D(G)\neq n(G)-2$. 

Assume finally that  $G^{\ast}$ is  a {\bf G}$_{10}$ graph.  If the two degree-$3$ adjacent vertices of type (K) have different sizes, then $G$ is an almost asymmetric graph and 
the result is immediate by Corollary \ref{lem-almost-asymmetric}. 
If the two degree-$3$ adjacent vertices of type (K) have the same size, then 
$D(G) = n(G)-2$ implies that some vertices of 
$G^{\ast}$ have size $0$ which is clearly not possible. 
\end{proof}

Before presenting the main result of this section, we need to recall the following theorem and deduce from it another lemma. Here, $\alpha(G^{\ast})$ denotes the number of vertices of $G^{\ast}$ of type (KN).

\begin{theorem} {\rm \cite[Theorem 2.14]{Hernando}}
\label{Hernando}
If $G$ is a connected graph with $d = \diam(G) \geq 3$, then $\dim(G) = n(G) - d$ if and only if $G^{\ast}$ is one of the following graphs.  
\begin{itemize}
\item[(a)] 
$G^{\ast}= P_{d+1}$ and one of the following cases holds: 
\begin{itemize} 
\item[(a1)] 
$\alpha(G^{\ast}) \leq 1$;
\item[(a2)]
$\alpha(G^{\ast}) = 2$, the two vertices of $G^{\ast}$ not of type (1) are adjacent, and if one is a leaf of type (K), then the other is also of type (K);
\item[(a3)] 
$\alpha(G^{\ast}) = 2$,
the two vertices of $G^{\ast}$ not of type (1) are at distance $2$ and both are
of type (N); or
\item[(a4)]
$\alpha(G^{\ast}) = 3$ and there is a vertex of type (N) or (K) adjacent to two vertices of type (N).
\end{itemize}

\item[(b)] 
$G^{\ast}= P_{d+1, k}$ (the path $(u^{\ast}_0, u^{\ast}_1, \ldots , u^{\ast}_d)$ with one extra vertex adjacent to $u^{\ast}_{k-1}$)
for some integer $k \in [3, d - 1]$, the degree-$3$ vertex $u^{\ast}_{k-1}$ of $G^{\ast}$ is of any type, each
neighbour of $u^{\ast}_{k-1}$ is of type (1N), and every other vertex is of type (1).

\item[(c)] 
$G^{\ast}= P'_{d+1, k}$ (the path $(u^{\ast}_0, u^{\ast}_1, \ldots , u^{\ast}_d)$ with one extra vertex adjacent to $u^{\ast}_{k-1}$ and 
$u^{\ast}_{k}$) for some integer $k \in [2, d - 1]$, the three vertices in the cycle are of type
(1K), and every other vertex is of type (1).
\end{itemize}
\end{theorem} 

\begin{lemma}\label{Hernando-Lem}
Let $G$ be a connected graph with $d = \diam(G) \in \{3, 4\}$. If $\dim(G) = n(G) - d$, then $D(G)=n(G)-2$ if and only if $G=P_4$. 
\end{lemma}

\begin{proof} 
The graph $G^{\ast}$ is one of the graphs described in Theorem \ref{Hernando}. 

In the case $(a1)$, assume first that $d=3$.
If $\alpha(G^{\ast})=0$, then $G=P_4$ and $D(G)=n(G) - 2$. 
In what follows, we will see that $G=P_4$ is the only graph with $D(G)=n(G)-2$ satisfying the assumptions of this lemma.
If $\alpha(G^{\ast})=1$, then there are two non-isomorphic positions for placing a vertex of type (K) or (N) in $P_4$. In all cases, the distinguishing number is equal to $n-3$. 
(Note that if $d=4$, then $\dim(G)=n(G)-4$ and by Proposition~\ref{prop:main} we must have $D(G)\leq n(G)-3$. 
Therefore, when $d=4$, it is impossible to have $D(G)= n(G)-2$ while $\dim(G)=n(G)-4$.)

In all the other cases, that is, $(a2)$, $(a3)$, $(a4)$, $(b)$, and $(c)$, there are no graphs with distinguishing number $n(G)-2$, so we can proceed by using 
Corollary~\ref{lem-almost-asymmetric} for the graph $G^{\ast}$ with $n(G^{\ast})\geq 5$, and techniques similar to those used in the proofs of the previous lemmas for the graph $G^{\ast}$ with $n(G^{\ast})=4$.
\end{proof}

We have now arived to the main result of this section. 

\begin{theorem}
If $G$ be a graph with $n(G)\geq 4$, then $D(G)=n(G)-2$ if and only if $G$ is one of the following graphs: 
\begin{itemize}
\begin{small}
\begin{multicols}{2}
\item[(1)] 
$C_5$ 
\item[(3)] 
$K_{1,2,2}$  
\item[(5)] 
$K_{3,3}$
\item[(7)] 
$K_{t,2}$, $t\geq 3$ 
\item[(9)] 
$K_{2}+ \overline{K_t}, t\geq 2$  
\item[(11)] 
$K_t + \overline{K_2}$, $t\geq 2$ 
\item[(13)] 
$K_1 + (K_t \cup K_1)$, $t\geq 2$
\item[(2)] 
$P_4$
\item[(4)] 
$2K_2 \cup K_1$ 
\item[(6)] 
$2K_3$ 
\item[(8)] 
$K_{t}\cup K_2$, $t\geq 3$ 
\item[(10)] 
$K_t \cup 2K_1$, $t\geq 2$ 
\item[(12)] 
$\overline{K_t}\cup K_2$, $t\geq 2$  
\item[(14)] 
$K_{t, 1}\cup K_1$, $t\geq 2$\,. 
\end{multicols} 
\end{small}
\end{itemize}
\end{theorem}

\begin{proof}
Let $G$ be a connected graph with $D(G)=n(G)-2$. Proposition \ref{prop:main}, Theorem \ref{t2}, and Corollary~\ref{cor:r1} imply that it suffices to check the graphs which have metric dimension equal to $n(G)-2$ or to $n(G)-3$. 

Assume first that $\dim(G)=n(G)-2$ and set $\ell=2$ in Proposition~\ref{prop:lemn-2}. 
The graphs in (a), (b), (c), (d), and (e) in Proposition~\ref{prop:lemn-2} are the graphs in (5), (7), (9), (11), and (13), respectively. Note that the graph in (f) in Proposition~\ref{prop:lemn-2} appears in (11). 

Assume second that $\dim(G)=n(G)-3$. Since $\dim(G) \leq n(G)- \diam(G)$, we have ${\rm diam}(G)\leq 3$. 
If ${\rm diam}(G)=2$, then the graphs in (1), (3) and (9) follow directly from Theorem~\ref{Jannesari} and Lemmas~\ref{G1}--\ref{Gal}.
If ${\rm diam}(G)=3$, then Theorem~\ref{Hernando} and Lemma~\ref{Hernando-Lem} imply that $G$ is the graph in (2). If $G$ is disconnected, then by Lemma~\ref{lem:discon}, $G$ has the form described in (4), (6), $\ldots ,$ (14) as the complements of graphs in (3), (5), $\ldots ,$ (13), respectively. Note here that $C_5$ and $P_4$ are self-complementary graphs.
\end{proof}

To conclude the paper we note that recently graphs $G$ with $\dim(G) = n(G) - 4$ have been described~\cite{dim n-4}. Using the above approach we believe that it is also possible to classify the graphs $G$ with $D(G)=n(G)-3$.  

%%%%%%%%%%%%%%%%%%%%%%%%%%%%%%%%%%%%%%%%%%%%%%%%%%%%%%%
\section*{Acknowledgments}
%%%%%%%%%%%%%%%%%%%%%%%%%%%%%%%%%%%%%%%%%%%%%%%%%%%%%%% 
The first two authors have been supported by the Discrete Mathematics
439 Laboratory of the Faculty of Mathematical Sciences at Alzahra University. 
Sandi Klav\v{z}ar was supported by the Slovenian Research Agency (ARIS) under the grants P1-0297, N1-0355, and N1-0285.

%%%%%%%%%%%%%%%%%%%%%%%%%%%%
\section*{Declaration of interests}
%%%%%%%%%%%%%%%%%%%%%%%%%%%%

The authors declare that they have no known competing financial interests or personal relationships that could have appeared to influence the work reported in this paper.

%%%%%%%%%%%%%%%%%%%%%%%%%%%%
\section*{Data availability}
%%%%%%%%%%%%%%%%%%%%%%%%%%%%

Our manuscript has no associated data.

\end{document}